\documentclass[12pt]{amsart}
\usepackage{txfonts}
\usepackage{bbm}
\usepackage{mathrsfs}
\usepackage{amssymb}
\usepackage{amssymb,amsmath,amsthm,amscd}
\usepackage{color}
\usepackage[all]{xy}

\numberwithin{equation}{section}

\newtheoremstyle{fancy1}{10pt}{10pt}{\itshape}{12pt}{\textsc\bgroup}{.\egroup}{8pt}{
}
\newtheoremstyle{fancy2}{10pt}{10pt}{}{12pt}{\itshape}{.}{8pt}{ }

\theoremstyle{fancy1}

\newtheorem{lem}[equation]{Lemma}

\newtheorem{thm}[equation]{Theorem}
\newtheorem{problem}[equation]{Problem}
\newtheorem{main}{Theorem}
\newtheorem*{main*}{Theorem}

\newtheorem*{cor*}{Corollary}

\theoremstyle{fancy2}

\newtheorem{rem}[equation]{Remark}
\newtheorem*{rem*}{Remark}

\newcommand{\cref}[1]{Corollary~\ref{#1}}

\newcommand{\Sph}{\mathbb{S}}

\newcommand{\R}{{\mathbb{R}}}
\newcommand{\Z}{{\mathbb{Z}}}

\newcommand{\G}{\ensuremath{\operatorname{\mathsf{G}}}}
\newcommand{\D}{\ensuremath{\operatorname{\mathsf{D}}}}

\renewcommand{\S}{\ensuremath{\operatorname{\mathsf{S}}}}

\def\con#1=#2(#3){#1 \equiv #2 \bmod{#3}}

\newcommand{\diam}{\ensuremath{\operatorname{diam}}}
\newcommand{\rad}{\ensuremath{\operatorname{rad}}}

\newcommand{\curv}{\ensuremath{\operatorname{curv}}}

\newcommand{\ric}{\ensuremath{\operatorname{ric}}}
\newcommand{\vol}{\ensuremath{\operatorname{vol}}}

\newcommand{\dist}{\ensuremath{\operatorname{dist}}}

\newcommand{\no}{\noindent}

\begin{document}
\date{\today}

\title{The Boundary Conjecture for Leaf Spaces}

\author{Karsten Grove}
\address{Department of Mathematics\\
University of Notre Dame\\
      Notre Dame, IN 46556\\USA
      }
\email{kgrove2@nd.edu }

\author{Adam Moreno}
\address{Department of Mathematics\\
University of Notre Dame\\
      Notre Dame, IN 46556\\USA
      }
\email{amoreno3@nd.edu}

\author{Peter Petersen}
\address{Department of Mathematics\\
UCLA\\
      Los Angeles, CA 90095\\USA
      }
\email{petersen@math.ucla.edu}

\thanks{The first named author was supported in part by a grant from the NSF}

\begin{abstract}
We prove that the boundary of an orbit space or more generally a leaf space of a singular Riemannian foliation is an Alexandrov space in its intrinsic metric, and that its lower curvature bound is that of the leaf space. A rigidity theorem for positively curved leaf spaces with maximal boundary volume is also established and plays a key role in the proof of the boundary problem.
\end{abstract}

\maketitle

%-------------- Article Text--------------------

%%%%%%%%%%%%%%%%%%%%%%%%%%%%%%%%%%%%%%%%%
%%%%%%%%%%%%%%%%%%%%%%%%%%%%%%%%%%%%%%%%%
%%%%%%%%%      Introduction   %%%%%%%%%%%%%
%%%%%%%%%%%%%%%%%%%%%%%%%%%%%%%%%%%%%%%%%
%%%%%%%%%%%%%%%%%%%%%%%%%%%%%%%%%%%%%%%%%
A basic conjecture going back to the early days of Alexandrov geometry states that the boundary $\partial X$ of an Alexandrov space $X$, with its induced length metric, is itself an Alexandrov space  with the same lower curvature bound as that of $X$.

In case $X$ is a convex subset of a Riemannian manifold $M$ with lower bound on sectional curvature this conjecture was verified in \cite{AKP}.

In this note we deal with other important classes of Alexandrov spaces from Riemanian geometry, namely, \emph{orbit spaces} of isometric group actions or more generally \emph{leaf spaces} of \emph{singular Riemannian foliations} (to be abbreviated as SRF).

\begin{main}
Let $M$ be a closed Riemannian manifold, and $\mathcal{F}$ a singular Riemannian foliation on $M$. If all the leaves in $\mathcal{F}$ are closed, then the boundary conjecture holds for the Alexandrov space $X = M/\mathcal{F}$.
\end{main}

Our proof is intertwined with that of the following \emph{rigidity} result for leaf spaces  $X = M/\mathcal{F}$ with curv$X \ge 1$. Here $\Sph^n$ denotes the unit sphere in euclidean $(n+1)$-space and the spherical join
\begin{center}
$L = L^n_{\alpha} = \Sph^{n-2} * [0,\alpha]$, $0 < \alpha \le \pi$
\end{center}

\no is referred to as the $n$-dimensional \emph{lens} with angle $\alpha$.

\begin{main}
If $X = M/\mathcal{F}$ be an $n$-dimensional leaf space with $\curv X \ge 1$ and $\vol (\partial X) = \vol (\Sph^{n-1})$, then $X$ is isometric to a lens, $L^n_{\alpha}$ with angle $\alpha = \pi/k$, where $k\in \Z_+$.
\end{main}

Here any such lens will arise. In fact the following construction exhibits essentially all Riemannian manifolds $M$ with a SRF $\mathcal{F}$ and leaf space isometric to a lens:
\smallskip

$\bullet$ Construction: Let $\mathcal{F}$ be a SRF on $\Sph^m$, $m \ge 1$ with $\dim \Sph^m/\mathcal{F} = 1$. In particular, $\mathcal{F}$ is the trivial  point foliation when $m=1$. From \cite{Mu} we know that $\Sph^m/ \mathcal{F} = [0, \alpha]$ with $\alpha \in \{\pi, \pi/2, \pi/3, \pi/4, \pi/6\}$ when $m \ge 2$ and, of course, $\Sph^m/ \mathcal{F} = \Sph^1$ when $m=1$.

Now let $\G$ be a compact Lie group acting isometrically on $\Sph^m$. Further assume that it leaves $\mathcal{F}$ invariant in such a way that the identity component $\G_0$ of $\G$ preserves each leaf $F \in \mathcal{F}$. In this case $\G/\G_0$ is either trivial or $\Z_2$ when $m \ge 2$. If $m=1$, then $\G$ is clearly finite and we assume $\Sph^1/\G = [0,\alpha]$. In particular, $\G = \D_k$ is a dihedral group of order $2k$ and $\alpha = \pi/k$ with $k \in \Z_+$.

Now let $\mathcal{F}$ also denote its canonical ``join'' extension to $\Sph^{n+m-1} = \Sph^{n-2}* \Sph^m$ with point leaves on $\Sph^{n-2}$. Furthermore, assume $\G$ acts freely and by isometries on a Riemannian manifold $P$. Clearly, $\G$ preserves the SRF on $P \times \Sph^{n+m-1}$ with leaves $P \times L$, $L \in\mathcal{F}$ and hence induces a SRF on the associated  bundle $M = P \times_{\G} \Sph^{n+m-1}$ with leaf space $L^n_{\alpha}$, where $\alpha$ is the length of the interval $(\S^m/\mathcal{F})/(\G/\G_0)$. 

\smallskip

We mention that a simple application of the slice theorem for SRF in \cite{MR} and the geometry of $L^n_{\alpha}$ in fact implies that any Riemannian manifold $M$ with a SRF $\mathcal{F}$ having leaf space $L^n_{\alpha}$ is foliated diffeomorphic to a manifold constructed as above, with the possible exception of $m  > 1$, $\alpha < \pi$ and $\dim M/\mathcal{F} = 2$.

\medskip
Both Theorem A and Theorem B are proved by induction on dimension. 

\smallskip

We refer to \cite{Pet}, respectively \cite{BGP} and \cite{BBI} for basic tools and results in Riemannian, respectively Alexandrov geometry that will be used freely. It is our pleasure to thank Marco Radeschi for constructive comments, and in general for sharing his insights on singular Riemannian foliations with us.

%%%%%%%%%%%%%%%%%%%%%%%%%%%%%%%%%%%%%%%%%%
%%%%%%%%%%%%%%%%%%%%%%%%%%%%%%%%%%%%%%%%%%
\section{Preliminaries}		%%%%%%%%%%%%%%%%%%%%%%%%%%%
%%%%%%%%%%%%%%%%%%%%%%%%%%%%%%%%%%%%%%%%%%%
%%%%%%%%%%%%%%%%%%%%%%%%%%%%%%%%%%%%%%%%

In this section we exhibit a few facts about  the Alexandrov geometry of leaf spaces needed in our proofs. When no curvature assumptions are made, these facts are merely translations to Alexandrov geometry of by now well known results for SRF presented in \cite{Ra}.

Throughout,  $M$ is a closed Riemannian manifold and $\mathcal{F}$ a SRF on $M$ with closed leaves. Since, by definition, leaves are locally everywhere equidistant from one another, the Riemannian distance between any pair of leaves agrees with the classical Haussdorff distance between them. When equipped with this metric, 

$\bullet$ $X= M/\mathcal{F}$ is a geodesic space of finite dimension.

$\bullet$ The projection map $P: M \to X=M/\mathcal{F}$ is a submetry, 

\no i.e., the image of any $r$-ball in $M$ is the corresponding $r$-ball in $X$. It follows, \cite{BGP}, that 

\smallskip
$\bullet$ $X= M/\mathcal{F}$ is an Alexandrov space.

\medskip
We now describe the space of directions $S_FX$ at $F \in X$, corresponding to a leaf $F \in \mathcal{F}$. Clearly all directions are geodesic directions. Associated to each leaf $F \in \mathcal{F}$ there is an infinitesimal SRF on each tangent space along $F$ which restricts to a SRF on each normal space, $T_F^{\perp}$ to $F$. This induced foliation is invariant under homotheties, and hence comes from a SRF,  $\mathcal{F}_{F^\perp}$, on the normal sphere $\Sph_F^{\perp}$ at any point $p\in F$. The (foliation) \emph{holonomy group}, $\G_F$ of the leaf $F$ at $p$
acts by isometries on $\Sph_F^{\perp}$ preserving  $\mathcal{F}_{F^\perp}$. Two leaves in $\mathcal{F}_{F^\perp}$ correspond to the same leaf in $\mathcal{F}$ if and only if they are in the same $\G_F$ orbit (see \cite{Ra}). It follows that

\smallskip
$\bullet$ $S_{F}X$ is isometric to $(\Sph_F^{\perp}/\mathcal{F}_{F^\perp})/\G_F$,

\medskip

The manifold $M$ as well as its leaf space $X = M/\mathcal{F}$ admits two natural \emph{stratifications}, one finer than the other (see \cite{Ra}). 

\no The coarser stratification is given by the dimensions of the leaves: 

For each $0 \le d \le \dim M$, the $d$-stratum $M_d \subset M$ is simply the union of all leaves $F \in \mathcal{F}$ with $\dim F = d$:

$\bullet$ each component of $M_d$ is a submanifold of $M$

\no and the closure satisfies:

$\bullet$  $\bar{M_d} = \cup_{d' \le d} M_{d'}$.

\no Furthermore, the restriction of 
$\mathcal{F}$ to any such component is a Riemannian foliation. As all leaves are compact, the restriction of $\mathcal{F}$ to each component of $M_d$ is locally a Riemannian  submersion whose image in $X=M/\mathcal{F}$ is intrinsically a \emph{Riemannian orbifold}. The \emph{regular} part, $M_{reg}$, of $M$ is the non-empty stratum $M_{d_0}$ with $d_0$ maximal. In particular, $M_{reg}$ is open, dense, and connected in $M$, moreover, clearly $\dim X = \dim M-d_0$.

The finer stratification amounts to further stratifying each component $X_d = M_d/\mathcal{F}$ according to its orbifold singularities described via the finite groups defining the local orbifold structure. Each component of the \emph{fine stratification} of $X$ is a manifold which is \emph{locally totally geodesic} in $X$. With this stratification, the ``singular'' points of $M_{reg}/\mathcal{F}$ correspond to so-called \emph{exceptional leaves}, whereas the non-singular points of $M_{reg}/\mathcal{F}$ correspond to \emph{principal leaves}. This part is also open, dense, and even convex in $X$.

For each $F \in \mathcal{F}$, the coarse stratum it belongs to is locally determined by the subspace $V_0 \subset T_F^{\perp}$ making up the \emph{point leaf stratum} of its infinitesimal foliation $\mathcal{F}_{F^\perp}$. The tangent space of its fine stratum at $F \in X$ is isomorphic to the fixed point set $V_0^{\G_F}$ of the holonomy group.

Of particular interest for us is the boundary $\partial X$ of $X$ when non-empty. Clearly, $\partial X$ is the closure of the $\dim X -1$ dimensional strata in $X$. Each component of a stratum of dimension $\dim X-1$ is an open subset of $\partial X$, referred to as an \emph{open face} of $\partial X$. An open face is closed in $\partial X$ only if it also constitutes a component of $\partial X$. The closure of an open face is simply called a \emph{face of the boundary}. In other words, 

$\bullet$ $\partial X$ is the union of its finitely many faces. 

%\no  $\spadesuit$\footnote{This definition runs on and on and is hard to break up, also not sure it helps} Alternatively, the points of an open face correspond to leaves $F$ where the associated infinitesimal foliation $\mathcal{F}_{F^\perp}$ on the normal space $T_p^{\perp}$ to say $p \in F$ have point leaves on the subspace $V \subset T_p^{\perp}$ corresponding to the $\spadesuit$\footnote{$F$-stratum or $\mathcal{F}$-strata?} $F$-strata, and perpendicular to $V$ the leaves are the round spheres centered at the origin of $V^{\perp}$, when $\dim V^{\perp} \ge 2$, or when $\dim V^{\perp} = 1$, the holonomy is $\Z_2$ acting by reflection with mirror $V$. In the latter case, the open face consists of exceptional leaves.

\medskip

For non-trivial infinitesimal SRF we have the following simple fact

\begin{lem}
Let $\mathcal{F}$ be a SRF with closed leaves on $\Sph^n$. Either $\diam \Sph^n/\mathcal{F} \le \pi/2$ or $\diam \Sph^n/\mathcal{F} = \pi$. In the latter case, $\mathcal{F}$ is a suspenion of a SRF on $\Sph^{n-1}$. In particular, the radius $\rad  \Sph^n/\mathcal{F} \le \pi/2$ in all cases.\end{lem}

\begin{proof}
Suppose $\diam \Sph^n/\mathcal{F} > \pi/2$ and let $F_1$ and $F_2$ be leaves with $\dist (F_1, F_2) = r > \pi/2$. By equidistance of leaves, the convex set $ C = \Sph^n - B(F_1, r)$ is a union of leaves of $\mathcal{F}$ and the point $s \in C$ at maximal distance to the boundary of $C$ is a leaf. Again by equidistance of leaves $-s$ is also a leaf and $\mathcal{F}$ is the suspension of its restriction to the equator of $\{s, -s\}$.
To prove the last statement we note that by induction it follows that $\Sph^n/\mathcal{F}$ is an iterated spherical suspension of an SRF of a sphere whose leaf space has $\diam \leq \pi/2$, and hence $\rad \Sph^n/\mathcal{F} = \pi/2$.
\end{proof}

\section{Rigid lens characterization}        %%%%%%%%%%%%%%%%%%%%%%
%%%%%%%%%%%%%%%%%%%%%%%%%%%%%%%%
%%%%%%%%%%%%%%%%%%%%%%%%%%%%%%%

A lens $L = L^n_{\alpha} = \Sph^{n-2} * [0,\alpha]$  with $\alpha \in (0,\pi]$ is clearly an  an Alexandrov space with curvature $\text{curv} L \ge 1$ and boundary $\partial L$ isometric to $\Sph^{n-1}$ when equipped with its intrinsic length metric.

Conversely, consider the following

\begin{problem}[Lens Problem]
Let $X$ be an $n$-dimensional Alexandrov space with $\text{curv} X \ge 1$ and boundary $\partial X$ isometric to $\Sph^{n-1}$, relative to its induced length metric. Is $X$ isometric to an $n$-dimensional lens? 
\end{problem}
 
In the special case when $X = N$ is a Riemannian manifold with smooth (convex) boundary $N$ is, indeed, isometric to the closed hemisphere of radius $1$, i.e., $N = L_{\pi}$ in our terminology. This is in fact a special case of the main theorem in \cite{HW}:
 
 \begin{thm}[Hang - Wang]\label{hw}
 Let $N$ be an $n$-dimensional compact Riemannian manifold with $\ric N \ge n-1$ and convex boundary $\partial N$, i.e., its second fundamental form is nonnegative. Then $N$ is isometric to a hemisphere of $\Sph^n$ provided $\partial N$ is intrinsically isometric to $\Sph^{n-1}$.
 \end{thm}
 
 As a precursor to Theorem B, we will provide a positive answer to the above problem when  $X = M/\mathcal{F}$ is the leaf space of a singular Riemannian foliation by closed leaves (for the general case cf. \cite{GP}). 
 
 We begin with the following
 
 \begin{lem}
 When $X = M/\mathcal{F}$ has $\curv X \ge 1$ and $\partial X = \Sph^{n-1}$, then $X - \partial X$ is a smooth manifold. Moreover, any $x \in X$ has distance at most $\pi/2$ to the soul point $s\in X$. In particular, $X = L^n_{\pi}$, when $|s \partial X| = \pi/2$.
 \end{lem}
 
 \begin{proof}
 We start by showing that the interior of $X$ consists of principal leaves. Assume that $x\in \mathrm{int}X$ is not a principal leaf. By Lemma 1.1 it follows that $\vol S_xX \leq \frac{1}{2} \vol \Sph^{n-1} $. Now glue the constant curvature 1 hemisphere $\mathbb{D}^{n}$ with boundary $\Sph^{n-1}$ onto $X$ along the boundary $\partial X$, to obtain an $n$-dimensional Alexandrov space $Y$ without boundary and with curv$Y \ge 1$ \cite{Pe1}. By volume comparison we have:
 \begin{eqnarray*} 
 \vol (\mathbb{D}^{n} ) & = & \frac{1}{2} \vol (\Sph^{n})\\ 
 & < & \vol Y\\ 
 & \le & \vol (\Sigma_1 S_xX)\\ 
 & = & \vol (S_xX)  \int_0^{\pi} {\sin^{n-1} (t )} dt\\ 
 & \leq & \frac{1}{2} \vol (\Sph^{n-1}) \int_0^{\pi} {\sin^{n-1} (t )} dt\\ 
 & = & \frac{1}{2} \vol (\Sph^{n}).
 \end{eqnarray*} 
  A contradiction.

   Let $s \in X$ be the soul of $X$, i.e., the unique point at maximal distance to $\partial X$. We claim that any point of $X$ has distance at most $\pi/2$ to $s$. To see this, we begin by showing that the distance from the boundary $\partial X$ to $s$ is at most $\pi/2$. In fact, note that, on one hand the boundary is convex, and on the other that, the distance function from s is a support function from below at any point on  $\partial X$ closest to s. If such a point was at distance $>\pi/2$ from $s$, the boundary would be concave rather than convex. Now, let $x \in X$ be any point of $X$ and $c$ a minimal geodesic from $s$ to $x$. Since $s$ is a critical point for the distance function to $\partial X$ there is a minimal geodesic, $d$ from $s$ to a closest point $y \in \partial X$ that makes an angle at most $\pi/2$ to $c$. By the choice of $y$ any minimal geodesic $e$ from $y$ to $x$ will make an angle at most $\pi/2$ to $d$ at $y$. By the Toponogov's comparison theorem both $e$ and $c$ have lengths at most $\pi/2$.

 Note that if $|s \partial X| = \pi/2$, then every $x \in \partial X$ has distance $\pi/2$ to $s$. It follows that $X$ is the spherical cone on $\partial X$, which in turn is isometric to the space of directions $S_sX = \Sph^{n-1}$ at $s$. 
 \end{proof}

 As the rest of the proof is by induction on dimension we start with $\dim X = 2$.
  
 \begin{thm}[Induction Anchor]
 Let $X = M/\mathcal{F}$ be a two-dimensional leaf space with $\text{curv} X \ge 1$ and boundary $\partial X$. The length of the boundary satisfies $\ell(\partial X) \le 2 \pi$ with equality if and only if $X =  L^2_{\alpha}$, where $\alpha = \pi/k$, for $k\in \Z_+$.
  \end{thm}
 
 \begin{proof}
 
 Recall that any two dimensional leaf space is an orbifold \cite{LT}. Moreover let $F$ be a vertex point $F \in \mathcal{F}$. Then $\mathcal{F}_{F^\perp}$ is either a point foliation when $\dim F^\perp = 2$ or else a foliation by  isoparametric hypersurfaces in the $\Sph_F^{\perp}$. In the first case 
$(\Sph_F^{\perp}/\mathcal{F}_{F^{\perp}})/\G_F = (\Sph_F^{\perp})/\G_F$, where $\G_F = \D_k$ is a dihedral group. In the second case, $\Sph_F^{\perp}/\mathcal{F}_{F^\perp} = [0,\alpha]$, with $\alpha \in \{\pi, \pi/2, \pi/3, \pi/4, \pi/6\}$ by \cite{Mu}. In this case, it is also possible that the holonomy is $\G_F = \Z_2$ and interchanges the two focal leaves of the isoparametric foliation, thus cutting these angles in half, i.e., adding $\pi/8$ and $\pi/12$. In particular, $\partial X$ is a geodesic polygon $P$, where the angle at each vertex is $\pi/k$ with $k \ge 2$. From $\curv X \ge 1$ it follows, that $P$ can have at most three vertices.

Partition $P$ further (if necessary) so that each part of $P$ is a minimal geodesic in $X$. Join each partition point of $P$ to the soul point $s$, by a minimal geodesic. In $\Sph^2$ draw all the corresponding comparison triangles adjacent to one another as in $X$ with common vertex $s_0 \in \Sph^2$ say at the south pole corresponding to $s$. 
 
 Note that, unless all comparison angles at $s_0$ are equal to the actual angles at $s$, this configuration of geodesic triangles will not close up. If so, the gap is joined by a circular arc of radius the length of the corresponding geodesics. From Toponogov's triangle comparison theorem, the region $C'$ described in $\Sph^2$ is convex, and thus has length at most $2 \pi$, i.e., the inequality is established.
 
 Now assume $\ell(\partial X) = 2 \pi$, and note that a convex region $C \subset \Sph^2$ has boundary of length $2 \pi$ if an only if $C$ is the intersection of two closed hemispheres, i.e., a bi-angle, or lens in our terminology. This immediately rules out three vertices of $P$. Moreover, by construction of $C'$ and $\ell(\partial X) = 2 \pi$, it follows that all angles at $s$ in $X$ agree with the corresponding angles at $s_0$ and that there is no gap between the first and last triangle. Since the lengths of the sides opposite the angles at $s$ and $s_o$ are equal by construction, the hinge rigidity version of Toponogov implies that the surfaces spanned by the triangles in $X$ are isometric to the ones in $C'$. It follows that $X$ is the intersection of two hemispheres. 
\end{proof}

We are now ready to complete the proof of our main result in this section:

\begin{thm}\label{lenserid}
An Alexandrov leaf space $X^{n}$ with $\text{curv} X \ge 1$ and boundary $\partial X$ intrinsically isometric to $\Sph^{n-1}$ is isometric to $L^n_{\alpha}$ with $\alpha = \pi/k$, where $k\in \Z_+$.
\end{thm}

\begin{proof}

Assume by induction that the Theorem holds in all dimensions $\le n$, and consider an $(n+1)$-dimensional leaf space $X$ whose boundary is intrinsically isometric to $\Sph^n$.

For each boundary point $x \in \partial X$, the space  of directions $S_xX$ has non-empty boundary intrinsically isometric to $\Sph^{n-1}$ by assumption. By the induction hypothesis, for each $x \in \partial X$ the space of directions $S_xX$ is isometric to $L^n_{\alpha_x}$ for $\alpha_x$ as in the Theorem.

Assume first that for all $x \in \partial X$, $\alpha_x = \pi$, in particular, $\partial X$ is a single stratum of $X = M/\mathcal{F}$. Thus $X$ is a smooth Riemannian manifold with sectional curvature $\ge 1$ and totally geodesic boundary and the result follows from \cite{HW}.

Now suppose $x \in \partial X$ has $S_xX$ isometric to $L^n_{\alpha_x}$ with $\alpha_x < \pi$. Any such $x$ belongs to an $(n-1)$-dimensional stratum of $X$, and the closures of any two such strata cannot meet due to the induction hypothesis. So such a stratum is a totally geodesic $\Sph^{n-1}$ in the boundary $\Sph^n$, and only one such stratum can exist. Thus, the boundary has exactly two faces, each intrinsically isometric to a hemisphere of curvature 1, meeting each other at an angle $\pi/k$, where $k \in \Z_+$. As all singularities lie in strata of codimension $ \leq 2$ it follows from Theorem 1.4 in \cite{LT} that $X$ is an orbifold. We claim that it is, in fact, a good orbifold.

Consider the metric space $Y$ obtained by gluing together $2k$ copies of $X$  along one face at a time. Consecutive copies of $X$ are reflections of each other along their common face.  The dihedral group $\D_k$ of order $2k$ acts by isometries on $Y$ with quotient space $X$. Next note that the local orbifold covers of $X$ are determined by reflections for points on the open faces and for the singular points by the linear action of $\D_k$ on $\Sph^{n-2}*\Sph^1$ that acts trivially on $\Sph^{n-2}$, the space of directions for the $(n-1)$-dimensional stratum. Thus the local orbifold covers of $X$ are metrically isometric to open sets of $Y$. This shows that $Y$ is a smooth Riemannian manifold. It also contains a totally geodesic copy of $ \Sph^n $, namely, the fixed point set of a reflection in $\D_k$. We can then use \cite{HW} again to see that $Y$ is a constant curvature sphere and the quotient is an Alexandrov lens. 
\end{proof}

%\begin{rem}$\spadesuit$\footnote{should we skip the remark? }  $\spadesuit$\footnote{Yes, we have by now explained what happens by giving all the examples. Maybe we can ref to Lytchak elsewhere?}
%When all leaves in $\partial X$ are singular, then the infinitesimal foliation $\mathcal{F}_{F^\perp}$ on $\Sph_F^{\perp}$ arises from isoparametric hypersurfaces and $\alpha \in \{\pi, \pi/2, \pi/3, \pi/4, \pi/6\}$. This happens for example when $M$ is simply connected \cite{Ly}, Theorem 1.6.$\spadesuit$\footnote{ I removed the comment that if $M$ is simply connected there are no exceptional leaves, with a reference to Lytchak Marco mentioned. Put it back in since I found it}

%When one open face consists of singular leaves and the other of exceptional leaves, then the infinitesimal foliation $\mathcal{F}_{F^\perp}$ on $\Sph_F^{\perp}$ arises from isoparametric hypersurfaces with isometric focal leaves and the holonomy $\Z_2$ switches the two focal leaves. In this case $\alpha \in \{\pi/2, \pi/4, \pi/6, \pi/8, \pi/{12}\}$.

%Finally, when all boundary points correspond to exceptional leaves any angle $\pi/k$ with $k \in \Z_+$ as we have seen.
%\end{rem}

%%%%%%%%%%%%%%%%%%%%%%%%%%%%%%%%%%%%%%%%%%
%%%%%%%%%%%%%%%%%%%%%%%%%%%%%%%%%%%%%%%%%
\section{Intrinsic convexity of the regular boundary}   %%%%%%%%%%%%%%%%%%%%%%
%%%%%%%%%%%%%%%%%%%%%%%%%%%%%%%%%%%%%%%%%%%
%%%%%%%%%%%%%%%%%%%%%%%%%%%%%%%%%%%%%%%%%

We say that a point $x \in \partial X$ of an $(n+1)$-dimensional Alexandrov space $X$ is (boundary) \emph{regular}, if its space of directions $S_xX$ is isometric to a lens $\Sph^{n-1}* [0,\alpha]$, with $\alpha \le \pi$. We denote this set of points by $\partial X_0$. Note that by Theorem 2.4 these are precisely points on the boundary where $\partial S_xX = \Sph^{n-1}$ and thus correspond to the regular points on the boundary provided the boundary is an Alexandrov space.

In a leaf space $X = M/\mathcal{F}$, the regular boundary points are dense in the boundary since they are the $n$- and $(n-1)$-dimensional strata of $X$ in $\partial X$. Moreover, the open faces of $X$, i.e., the $n$-dimensional strata of $X$, are Riemannian manifolds that are (locally) totally geodesic in $X$. In particular, the (open) faces are locally Alexandrov with the same lower curvature bound at that of $X$. From Petrunin's gluing theorem \cite{Pe1} it also follows that, if $x$ is boundary regular but not in an open face, then we intrinsically still have the same local curvature control near $x$ on the boundary.

By \cite{Pe2}, to complete the proof of Theorem A, it thus suffices to prove that the set $\partial X_0$ is boundary convex, i.e., any minimal curve in $\partial X$ connecting two points of $\partial X_0$ lies entirely in $\partial X_0$.

\begin{thm}
In a leaf space $X = M/\mathcal{F}$, the set $\partial X_0$ of regular points is intrinsically convex.
\end{thm}

\begin{proof}
Let $x,y \in \partial X_0$ and connect them by a minimal curve $c: [a,b] \to \partial X$. Pick any $t_0 \in (a,b)$.

Now we blow up $X$ at $c(t_0)$. In the limit $c$ becomes a line $\ell = \R$ in the tangent cone $T_{c(t_0)}X$ passing through the cone point $c(t_0)$ and belongs to its boundary. 
By induction, we can assume that the boundary of the space of directions $S_xX$ at $x$ is an Alexandrov space with curv$\partial S_xX \ge 1$ (having verified it dimension 2). We claim that the intrinsic diameter of $\partial S_{c(t_0)}X $ is $\pi$. Suppose there is a path $\gamma$ in $\partial S_{c(t_0)}X $ of length less than $\pi$ joining the opposite directions $u$ and $v$ of $\ell$ at $c(t_o)$. In this case $c$ would not be a minimal curve in the boundary. It follows that $\partial S_{c(t_0)}X $  is intrinsically isometric to a spherical suspension of the space of directions at $u$ (and at $v$). Now the set of regular boundary point is clearly open in the boundary, and hence along $c$. We claim it is closed along $c$ as well. Suppose $t_0$ is the first time where $c(t_0)$ not regular. Then the corresponding direction $u$ towards $x$ is regular, and hence $\partial S_{c(t_0)}X $ is intrinsically isometric to $\Sph^{n-1}$. A contradiction. 
\end{proof}

\begin{rem}
The statement in Theorem B now follows from \ref{lenserid} and Theorem A, since knowing that $\partial X$ is an Alexandrov space with $\curv X \ge 1$ and maximal volume implies that it is a round sphere of curvature 1.
\end{rem}

%%%%%%%%%%%%%%%%%%%%%%%%%%%%%%%%%%%%%%%%%%%%%%%                                         %%%%%%%%%%%%%%%%%%%%%
%%%%%%%%%%%%%%%%%%%%%%%%%%%%%%%%%%%%%%%%%%%%%%%%%%%%%%%%
\providecommand{\bysame}{\leavevmode\hbox    %%%%%%%%%%%%%%%%%%%%%%
to3em{\hrulefill}\thinspace}                                            %%%%%%%%%%%%%%%%%%%%%

\end{document}